\documentclass[12pt]{amsart}

\usepackage{pifont}
\usepackage{latexsym}
\usepackage{amsmath}
\usepackage{amsfonts}
\usepackage{amssymb}
\usepackage[american]{babel}
\usepackage[babel, final]{microtype}
\usepackage{amssymb}

\usepackage[pdftex, final]{graphicx}
\usepackage{caption}
\usepackage{subcaption}
\usepackage{pinlabel}
\usepackage{enumerate}
\usepackage{url}

\usepackage[pdftex, dvipsnames]{xcolor}
\usepackage[pdftex,%
  final,%
  colorlinks=true,%
  linkcolor=NavyBlue,%
  citecolor=NavyBlue,%
  filecolor=NavyBlue,%
  menucolor=NavyBlue,%
  urlcolor=NavyBlue,%
  bookmarks=true,%
  bookmarksdepth=3,%
  bookmarksnumbered=true,%
  bookmarksopen=true,%
  bookmarksopenlevel=2,%
]{hyperref}
\input{macros}

\newcommand*{\connsum}{\mathbin{\sharp}}
\newcommand*{\T}{\mathrm{T}}
\newcommand*{\Li}{\mathrm{L}}
\newcommand*{\tb}{\mathrm{tb}}
\newcommand*{\rot}{\mathrm{rot}}
\newcommand*{\M}{\mathrm{M}}
\newcommand*{\Sp}{\mathrm{S}}
\newcommand*{\sg}{\mathrm{sg}}
\newcommand*{\N}{\mathrm{N}}
\newcommand*{\B}{\mathrm{B}}
\newcommand*{\R}{\mathbb{R}}
\newcommand*{\xistd}{\xi_{std}}
\newcommand*{\V}{\mathrm{v}}
\newcommand*{\Z}{\mathbb{Z}}
\newcommand*{\K}{\mathrm{K}}
\newcommand*{\slk}{\mathrm{sl}}

\newcommand*{\TM}{\mathrm{TM}}
\newcommand*{\D}{\mathrm{D}}
\newcommand*{\Ho}{\mathrm{H}}
\newcommand*{\X}{\mathrm{X}}
\begin{document}
\title {Links in overtwisted contact manifolds}

\author[Rima Chatterjee]{Rima Chatterjee}
\address{Mathematisches Institut\\ Universit\"at zu K\" oln\\ Weyertal 86-90, 
  50931 K\"oln, Germany}
\email{\href{mailto:rchatt@math.uni-koeln.de}{rchatt@math.uni-koeln.de}}
\urladdr{\url{http://www.rimachatterjee.com/}}

\begin{abstract}
We prove that Legendrian and transverse links in overtwisted contact structures having overtwisted complements can be classified coarsely by their classical invariants. We further prove that any coarse equivalence class of loose links has support genus zero and constructed examples to show that the converse does not hold. 
\end{abstract} 
\maketitle
\section{Introduction}
\label{intro}
Knot theory associated to contact 3-manifolds has been a very interesting field of study. We say a knot in a contact 3-manifold is \emph{Legendrian} if it is tangent everywhere to the contact planes and \emph{transverse} if it is everywhere transverse. The classification of Legendrian and transverse knots has always been an interesting and difficult problem in contact geometry. Two Legendrian knots are said to be \emph{Legendrian isotopic} if they are isotopic through Legendrian knots. A knot or link type is said to be \emph{Legendrian simple} if it can be classified by its classical invariants up to Legendrian isotopy.  There are only a few knot types that are known to be Legendrian simple in $(\Sp^3,\xistd)$. For example topologically trivial knots in \cite{ef}, the torus knots and figure eight knots in \cite{eth} are all Legendrian simple. While there is no reason to believe all knots should be Legendrian simple, it has been historically difficult to prove otherwise. Chekanov \cite{chekanov} and independently, Eliashberg \cite{eli} developed invariants of Legendrian knots that show that $m(5_2)$ has Legendrian representatives that are not distinguised by their classical invariants.

 Since Eliashberg's classification of overtwisted contact structures \cite{eliashbergovertwist}, the study of overtwisted contact structures and the knots and links in them, has been minimal. However, in recent years overtwisted contact structures have played central roles in many interesting applications such as building achiral Lefchetz fibration \cite{ef}, near symplectic structures on 4-manifolds \cite{gay} and many more. Thus the overtwisted manifolds and the knot theory associated to them has generated significant interest. There are two types of knots/links in overtwisted contact structures, namely loose and non-loose (Also known as non-exceptional and exceptional respectively).
A link in an overtwisted contact manifold is loose if its complement is overtwisted and non-loose otherwise. The first explicit example of a non-loose knot is given by Dymara in \cite{dy}. In general, non-loose knots appear to be rare. It is still not known if every knot type has a non-loose representative. We have another notion of classification of knots and links in contact manifolds known as \emph{coarse equivalence}. We say knots/links are coarsely classified if they are classified up to orientation preserving contactomorphism, smoothly isotopic to the identity. Observe that, though classification by Legendrian isotopy and coarse equivalence are equivalent in $(\Sp^3,\xistd)$, they are not the same in general. Eliashberg and Fraser gave a coarse classification of Legendrian unknots in overtwisted contact structure in $\Sp^3$ \cite{ef}. Later, Geiges and Onaran gave a partial coarse classification of the non-loose left handed trefoil knots in \cite{geiona2} and  non-loose Legendrian Hopf links in \cite{geiona}. Recently, Matkovi\v{c} in \cite{matko} extended their result. Note that, this is still not a complete classification. Also, all of these classification results have been proved in overtwisted $\Sp^3$.

This paper studies links in all overtwisted contact manifolds. In \cite{et}, Etnyre proved that loose, null-homologous Legendrian and transverse knots can be coarsely classified by their classical invariants. In \cite{geiona2}, the authors proved that for the loose Hopf link, this classification result remains true. It turns out that Etnyre's work very naturally extends for every null-homologous loose links (By null-homologous here we mean every link component bounds a Seifert surface) which is the first theorem of this paper:

\begin{theorem}
\label{thm:Legmain}
Suppose $\Li_1$ and $\Li_2$ are two loose null-homologous Legendrian links with same classical invariants. Then, $\Li_1$ and $\Li_2$ are coarsely equivalent.
\end{theorem}
\begin{remark}
Here by a null-homologous link, we assume that every link component is null-homologous.
\end{remark}
The above theorem tells us that there is only a unique loose link with any fixed classical invariants in any overtwisted contact structure up to contactomorphism.
\begin{remark}
In an overtwisted contact manifold $(\M,\xi)$, classification up to contactomorphism and classification upto Legendrian isotopy are not equaivalent. Our result doesn't say anything about the Legendrian simpleness  of a loose link. Dymara in \cite{dy} proved that two Legendrian knots having same classical invariants in any contact 3-manifold $(\M,\xi)$ are Legendrian isotopic if if there exists an overtwisted disk disjoint from both of them. Later this result was strengthened by Ding--Geiges in \cite{Ding-Geiges} and further by Cahn-Chernov in \cite{Cahn-Chernov} . In spite of being a stronger notion of classification, unfortunately this does not apply to all loose knots.
\end{remark}
As a corollary we proved the following result for loose transverse links.
\begin{corollary}
\label{cor:transverse}
Suppose $\T$ and $\T'$ are two transverse loose null-homologous links with same classical invariants. Then $\T$ and $\T'$ are coarsely equivalent.
\end{corollary} 
In other words, there is a unique loose null-homologous transverse link with every component having a fixed  self-linking number up to contactomorphism.
\begin{remark}
In \cite{et}, the theorem was proved for null-homologous knots and it was hinted that these might be extended to non-null homologous knots using Tchernov's definition of relative rotation number and relative Thurston--Benniquin number \cite{tchernov} with some extra conditions on the underlying manifold. It seems plausible that the same idea can be extended for links as well.
\end{remark}
After classifying the Legendrian and transverse loose links, we associate a Legendrian link with a compatible open book decomposition of the manifold. First, we extended the definition of the support genus of a Legendrian knot defined in \cite{ona} to the support genus of a Legendrian link (this extension comes naturally) and proved the following theorem about coarse equivalence class of  loose Legendrian links.

\begin{theorem}
\label{thm:supportgenus}
Suppose $[\Li]$ denotes the coarse equivalence class of loose, null-homologous Legendrian links with in any contact $3$-manifold.Then $\sg([\Li])=0$.

\end{theorem}

The above result gives a generalization (weak) of Onaran's result.

Like non-loose knots, non-loose links appear to be rare. The above theorem suggests, if we can find a Legendrian link $\Li$ with $\sg(\Li)>0$ that will immediately tell us that $\Li$ is non-loose.  
We also show that the converse of the theorem is not true by constructing planar open books for non-loose links.
\begin{theorem}
\label{thm:examples}
There are examples of non-loose links with support genus zero.
\end{theorem}

Also, as a corollary we have a similar result for coarse equivalence class of loose transverse links.
\begin{corollary}
Suppose $[\T]$ be a coarse equivalence class of loose, null-homologous loose transverse links. Then $\sg[\T]=0$
\end{corollary}

\subsection{Organization of the paper}
The paper has been organized in the following way: In \fullref{sec:basics}, we discussed preliminaries of contact geometry and Legendrian knots followed by a discussion of Pontyragin-Thom construction for manifolds with boundary in \fullref{sec:homotopy}. In \fullref{sec:Legtheorem}, we proved \fullref{thm:Legmain} and \fullref{cor:transverse}. We conclude with a proof of \fullref{thm:supportgenus} and \fullref{thm:examples} in \fullref{sec:openbook}.
\vspace{-2 mm}
\subsection{Acknowledgement} I wish to express my deepest gratitude to my advisor Shea Vela-Vick. His guidance, support and motivation over the years have been greatly appreciated. I would also like to thank C.-M. Mike Wong for asking the question which leads to \fullref{thm:examples}. This research is partially supported by NSF Grant 1907654 and the SFB/TRR 191 ``Symplectic Structures in Geometry, Algebra and Dynamics, funded by the Deutsche Forschungsgemeinschaff (Project- ID 281071066-TRR 191)''.


\section{Basics on contact geometry}
\label{sec:basics}
In this section, we briefly mention the preliminaries of contact geometry and Legendrian knots. For more details the reader should check \cite{etknot}, \cite{etcontact} and \cite{etnyrelectures}.
\subsection{Contact structures}
A contact structure $\xi$ on an oriented $3$-manifold $\M$ is a nowhere integrable $2$-plane field and we call ($\M,\xi$) a contact manifold. We assume that the plane fields are co-oriented, so $\xi$ can be expressed as the kernel of some global one form $\alpha$. In this case, the non-integrability condition is equivalent to $\alpha\wedge d\alpha > 0$. 
There are two types of contact structures--tight and overtwisted.
An overtwisted disk is a disk embedded in a contact manifold $(\M,\xi$) such that $\xi$ is tangent to the boundary of the disk. We call a contact manifold overtwisted, if it contains an overtwisted disk. Otherwise we call it tight.

Though only few results are knows about classifying tight contact structures on manifolds, overtwisted contact structures are completely classified by Eliashberg.
\begin{theorem}{(Eliashberg, \cite{eliashbergovertwist})}Two overtwisted contact structures are isotopic if and only if they are homotopic as plane fields. Moreover, every homotopy class of oriented 2-plane field contains an overtwisted contact structure.
\end{theorem}
\subsection{Legendrian links}
A link $\Li$ smoothly embedded in $(\M,\xi)$ is said to be Legendrian if it is everywhere tangent to $\xi$. 
For the purpose of this paper, by classical invariants of a link we refer to the classical invariants of its components. 
 The classical invariants of a Legendrian knot are the topological knot type, \emph{Thurston--Benniquin invariant} $\tb(\Li)$ and \emph{rotation number} $\rot(\Li)$.
$\tb(\Li)$ measures the twisting of the contact framing relative to the framing given by the Seifert surface of $\Li$. 
The other classical invariant $\rot(\Li)$ is defined to be the winding of $\T\Li$ after trivializing $\xi$ along the Seifert surface. 
One can classify a Legendrian link up to Legendrian isotopy. Two Legendrian links $\Li$ and $\Li'$ are said to be \emph{ Legendrian isotopic} if they are isotopic through Legendrian links.
 There is another type of classification of Legendrian links known as \emph{coarse equivalence}. We say two Legendrian links are \emph{coarsely classified} if they are classified up to orientation preserving contactomorphism, isotopic to the identity. In $(\Sp^3, \xistd)$ these two types of classification are equivalent . But in general a coarse equivalence does not imply Legendrian isotopy. 

\begin{figure}[!htbp]
\centering
\includegraphics[scale=0.2]{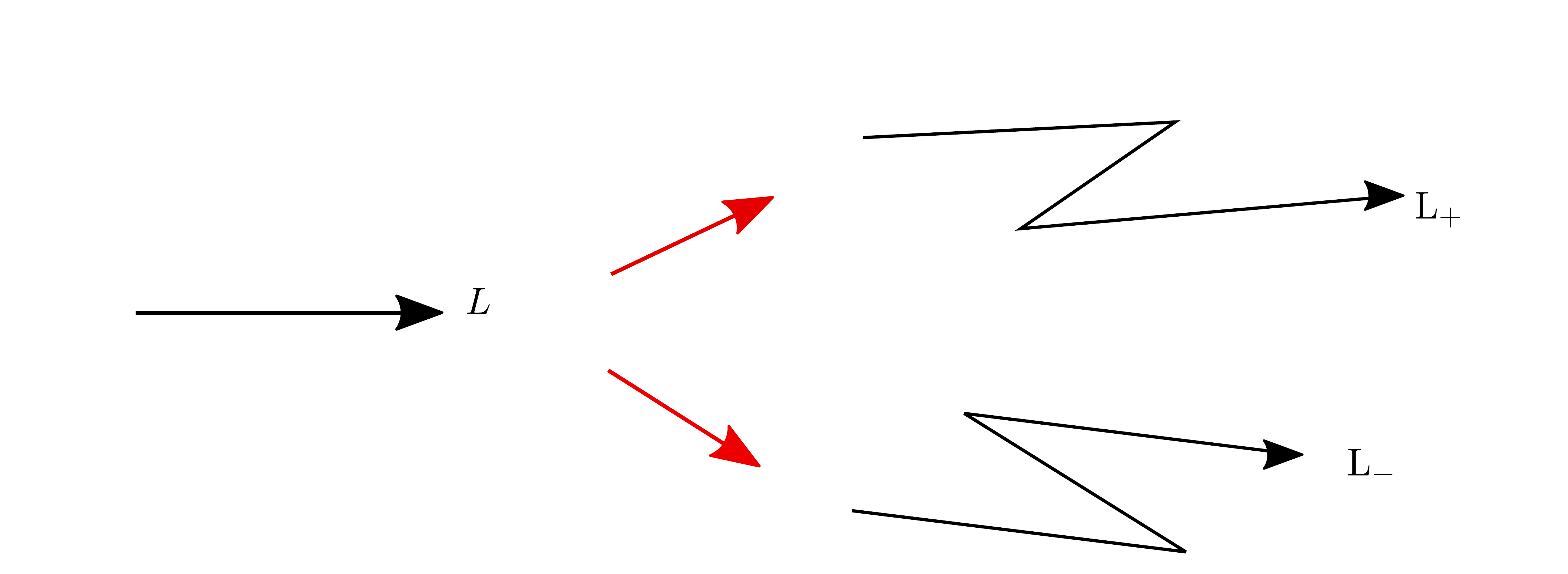}
\caption{Stabilizations of a Legendrian knot.}
\label{fig:stabilization}
\end{figure}
Stabilization of a link can be done by stabilizing any of the link component.
By standard neighborhood theorem of the Legendrian knot, one can identify any Legendrian link component $\Li$ locally with the $x$ axis. Stabilization is a local operation as shown in \fullref{fig:stabilization}. The modification on the top right-side is called the positive stabilization and denoted as $\Li_+$. The modification on the bottom right-side is known as negative stabilizations and denoted as $\Li_-$. It does not matter which order the stabilizations are being done, it just matters where those are being done. The effect of the stabilizations on the classical invariants are as follows:
\[\tb(\Li_\pm)=\tb(\Li)-1 \quad \text{and}\quad \rot(\Li_\pm)=\rot(\Li)\pm 1.\]
\subsection{Transverse link and its relationship with a Legendrian link}
A link $\T$ in ($\M,\xi$) is called transverse (positively) if it intersects the contact planes transversely with each intersection positive. By classical invariant of a transverse link, we will refer to the classical invariants of its components.  There are two classical invariants for transverse knot, the topological knot type and the {\it self-linking number} $\slk(\T)$. Self-linking number is defined for null-homologous knots. Suppose $\Sigma$ be a Seifert surface of a transverse knot. As $\Sigma|_\xi$ is trivial, we can find a non-zero vector field $v$ over $\Sigma$ in $\xi$. Let $\T'$ be a copy of $\T$ obtained by pushing $\T$ slightly in the direction of $v$. The self-linking number $\slk(\T)$ is defined to be the linking no of $\T$ with $\T'$.

Legendrian and transverse links are related by the operations known as transverse push off and Legendrian approximation. The classical invariants of a Legendrian link component and its transverse push off are related as follows:
\[ \slk(\Li_\pm)=\tb(\Li)\mp\rot(\Li)\] where $\Li_\pm$ denotes the positive and negative transverse push offs.
 In this paper, if we mention transverse push-off it is always the positive transverse pushoff unless explicitly stated otherwise. Note that, while a transverse push off is well defined, a Legendrian approximation is only well defined up to negative stabilizations. 
\subsection{Open book decomposition and supporting contact structures}
Recall an \emph{open book decomposition} of a $3$-manifold $\M$ is a triple $(\B,\Sigma,\phi)$ where $\B$ is a link in $\M$ such that $\M\setminus\B$ fibers over the circle with fiber $\Sigma$ and monodromy $\phi$ so that $\phi$ is identity near the boundary and each fiber of the fibration is a Seifert surface for $\B$. By saying $\phi$ is the monodromy of the fibration we mean that $\M\setminus\B= \Sigma\times[0,1]/\sim$ where $(1,x)\sim(0,\phi(x))$. The fibers of the fibration are called \emph{pages} of the open book and $\B$ is called the \emph{binding}. Given an open book $(\B,\Sigma,\phi)$ for $\M$, let $\Sigma'$ be $\Sigma$ with a $ 1$-handle attached. Suppose $c$ is a simple closed curve that intersects the cocore of the attached 1-handle exactly once. Set $\phi'=\phi\circ D_c^+$, where $D_c^+$ is a right handed Dehn-twist along $c$. The new open book $(\B',\Sigma',\phi')$ is known as the \emph{positive stabilization} of $(\B,\Sigma,\phi)$. If we use $D_c^-$ instead, that will be called a \emph{negative stabilization}. For details check \cite{etplanar}.

 We say a contact structure $\xi=\ker\alpha$ on $\M$ is supported by an open book decomposition ($\B, \Sigma,\phi)$ of $\M$ if
 \begin{enumerate}
 
  \item  $d\alpha$ is a positive area form on the page of the open book.
  \item  $\alpha(v)>0$, for each oriented tangent vector to $\B$.
  \end{enumerate}
Given an open book decomposition of a 3-manifold $\M$, Thurston and Winkelnkemper \cite{thurston} showed how one can produce a compatible contact structure. Giroux proved that two contact structures which are compatible with the same open book are isotopic as contact structures \cite{giroux}. Giroux also proved that two contact structures are isotopic if and only if they are compatible with open books which are related by positive stabilizations.

It is well known that every closed oriented 3-manifold has an open book decomposition. We can perform an operation called \emph{Murasugi sum} to connect sum two open books and produce a new open book. An interested reader should check \cite{etnyrelectures} for details.
\section{Homotopy classes of 2-plane fields}
\label{sec:homotopy}
In this section, we review the homotopy theory of plane fields in the complement of a link. Specifically, we will study homotopy classes of 2-plane fields on manifolds with boundary. We start by recalling, Pontyragin-Thom construction associated with manifolds with boundary (For Pontyragin-Thom construction for closed manifolds see \cite{milnor})
\subsection{Pontyragin-Thom construction for manifolds with boundary}
 Suppose $\M$ be an oriented manifold with boundary. The space of oriented plane-fields on $\M$ will be denoted as $\mathcal{P}(\M)$. On the other hand, if $\eta$ is a plane-field defined on the boundary of $\M$, then the set of all plane fields that extend $\eta$ to all of $\M$ will be denoted by $\mathcal{P}(\M,\eta)$. $\mathcal{V}(M)$ will be the set of all unit vector fields and $\mathcal{V}(M,v)$ will denote the set of all unit vector fields which extend $v$ to all of $\M$. Here $v$ is the unit vector field defined along $\partial M$. Also observe the sets $\mathcal{P}(\M,\eta)$ and $\mathcal{V}(\M,v)$ can be empty depending on $\eta$ and $v$.
 
 After choosing a Riemannian metric on $\M$ we can associate a unit vector field to an oriented plane field in the following way: We send a unit vector field $v$ to the plane field $\eta$ such that $v$ followed by the oriented basis of $\eta$ orients $\TM$. Thus there is a one-to-one correspondence between $\mathcal{P}(\M)$ and $\mathcal{V}(M)$. Similarly for $\mathcal{P}(\M,\eta)$ and $\mathcal{V}(\M,v)$ where $v$ is the unit vector field along the boundary associated to $\eta$ by a choice of metric and orientation. Notice both the correspondences only depend on a choice of metrics.
 
 We know that any 3-manifold has trivial tangent bundle. Thus fixing some trivialization we can write $\TM\simeq\M\times\R^3$. So the unit tangent bundle $\mathrm{UTM}$ can be identified with $\M\times\Sp^2$. Any unit vector field on $\M$ can be defined as a section of this bundle and can be associated to a map $\M\to\Sp^2$. We can identify $\mathcal{V}(\M)$ with $[\M, \Sp^2]$. Similarly if $v$ is a unit vector field on $\partial M$, we can associate  it with a map $f_v\colon\partial\M\to{\Sp}^2$. Thus $\mathcal{V}(\M,v)$ can be identified with the maps from $\M$ to ${\Sp}^2$ which coincides with $f_v$ on the boundary, denoted by $[\M,\Sp^2;f_v]$.
 
 Now Suppose $f_v\colon\partial M\to\Sp^2$ misses the north pole $p$. Now given any $f\in[\M, \Sp^2;f_v]$ we can homotope it so that it is transverse to the north pole (Thus $p$ will be a regular value for $f$). Then $f^{-1}(p)=\Li_f$ will be in the interior of $\M$ with framing $\bf{f}_f $ given by $f^*(T\Sp^2|_p)$. As $f$ homotopes through maps in $[\M,\Sp^2;f_v]$ the link $(\Li_f,\bf{f}_f)$ changes by framed cobordism. Thus any $v$ defined on $\partial \M$ which extends to $\M$ can be associated to framed cobordism classes of link. This gives us the relative version of Pontyragin-Thom construction.
 \begin{remark}
 Notice, this construction works fine if $\M$ has multiple boundary components.
 \end{remark}
 \begin{lemma}
\label{lemma:pont}
Assume that $\eta$ is a plane field defined along the boundary of $\M$ that in some trivialization of $\mathrm{TM}$ corresponds to a function that misses the north pole of $\Sp^2$. There is a one-to-one correspondence between homotopy classes of plane fields on $\M$ that extend $\eta$ on $\M$ and the set of framed links in the interior of $\M$ up to framed cobordism.
\end{lemma}
For the closed case, the following proposition was proved in \cite{gompf}.

\begin{proposition}
Let $\M$ be a closed, connected 3-manifold. Then any trivialization $\tau$ of the tangent bundle of $\M$ determines a function $\Gamma_\tau$ sending homotopy classes of oriented $2$-plane fields $\xi$ on $\M$ into $\Ho_1(\M,\Z)$ and for any $\xi$, $2\Gamma_\tau(\xi)$ is Poincar\'{e} dual to $c_1(\xi)\in\Ho^2(\M,\Z)$. For any fixed $x\in \Ho_1(\M,\Z)$, the set $\Gamma^{-1}(x)$ of classes of 2 plane-fields $\xi$ mapping to $x$ has a canonical $\Z$ action and is isomorphic to $\Z/d(\xi)$, where $d$ is the divisibility of the chern class.
\end{proposition}
Now suppose $\M$ is a manifold with boundary and $\mathcal{F}(\M)$ denotes the set of all cobordism classes of framed link in the interior of $\M$. Then there is a homomorphism 
\[\phi\colon\mathcal{F}\to\Ho_1(\M,\Z)\]
such that
\[ (\Li_f,f)\to [\Li].\] This map is clearly surjective. We want to compute the preimage of this map. First notice, there is a natural intersection pairing between $\Ho_1(\M)$ and $\Ho_2(\M,\partial M)$. Let $i\colon(\M,\emptyset)\to(\M,\partial\M)$ induces the map $i_*\colon\Ho_2(\M,\Z)\to\Ho_2(\M,\partial \M,\Z)$. For $\Li\in\Ho(\M,\Z)$, set
\[D_\Li=\{\Li\cdot[\Sigma]:\ {\text where}\ \Sigma\in i_*(\Ho_2(\M,\Z))\}\]
where $\Li\cdot\Sigma$ denotes the intersection pairing. Clearly this is a subset of $\Z$. Suppose $d(\Li)$ is the smallest non-negative integer in $D_\Li$.
\begin{lemma}
With the notations above, \[\phi^{-1}(\Li)=\Z/d(2\Li).\]
\end{lemma}

\section{Classification of loose Legendrian links}
\label{sec:Legtheorem}
There are two types of links in an overtwisted contact manifold, namely loose (also known as non-exceptional) and loose (also known as exceptional). A Legendrian link $\Li$ is called loose if the contact structure restricted to its complement is overtwisted. Otherwise, it is called non-loose. In other words, a loose link must have an overtwisted disk disjoint from it.
\begin{remark}
Note that, for a loose Legendrian link, all of its components must be loose. But a non-loose link can have loose components. In fact, a non-loose link can have all its components loose.
\end{remark}
The following is our main theorem in this section.
\begin{theorem}
 \label{thm:main}
 Suppose $\Li$ and $\Li'$ are two Legendrian $n$-component links in $(\M,\xi)$ with all of their components null-homologous. We fix their Seifert surfaces. If $\Li$ and $\Li'$ are topologically isotopic, $\tb(\Li_i)=\tb(\Li_i')$ and $\rot(\Li_i)=\rot(\Li_i')$ for $i=1\dots n$ (where the classical invariants are defined using the fixed Seifert surfaces), then $\Li$ and $\Li'$ are coarsely equivalent.
 \end{theorem}
 In other words, there is a unique loose Legendrian link with the components having fixed $\tb$ and $\rot$  up to contactomorphism.
  Before we begin proving this, we need the following lemma:

 \begin{lemma}
 \label{lemma:homotopic}
  Suppose $\Li$ and $\Li'$ be two Legendrian $n$-component links in $(\M,\xi)$ with each of their components being null-homologous. Suppose they are topologically isotopic, $\tb(\Li_i)=\tb(\Li_i')$ and $\rot(\Li_i)=\rot(\Li_i')$ for $i=1\dots n$, then $\xi|_{\M\setminus \N(\Li)}$ is homotopic to  $\xi|_{\M\setminus \N(\Li')}$ rel boundary as plane fields.
 \end{lemma}
 \begin{proof}
 We will use techniques similar to \cite{et}. 
 As $\Li$ and $\Li'$ are topologically isotopic, there is an ambient isotopy of $\M$ which takes $\Li$ to $\Li'$. We will assume that the Seifert surfaces of the link components are also related by this ambient isotopy (So after applying the ambient isotopy we assume the Seifert surfaces of the components agree).
 
 As $\Li$ and $\Li'$ are topologically isotopic there is an ambient isotopy of $\M$, $\phi_t$ such that $\phi_0=id$ and $\phi_1(\Li)=\Li'$. Using this isotopy we push forward the underlying contact structure $\xi$. Thus we now have a new contact structure ${\phi^{-1}_{1*}}{\xi}$ and call it $\xi'$. Observe $\xi$ and $\xi'$ are homotopic as plane fields in $\M$. After we apply the isotopy we can assume $\Li=\Li'$ and $\N$ be their standard neighborhood. Note that, $\tb$ measures the twisting of the contact framing with respect to the surface framing. As the components have the same $\tb$, this allows us to identify the neighborhoods. Now by standard neighborhood theorem of Legendrian links, $\xi$ and $\xi'$ agree on $\N$. We need to show that $\xi|_{\M\setminus\N}$ is homotopic to $\xi'|_{\M\setminus \N}$ rel boundary as plane fields. We know that homotopy class of plane fields are in one-to-one correspondence with  framed links up to framed cobordism. Now using Pontyragin-Thom construction for manifolds with boundary, we will associate these plane fields with $(\Li_\xi,f_\xi)$ and $(\Li_{\xi'},f_{\xi'})$. We need to show that these links are homologous in $\M\setminus \N$ and that their framing differs by $2d[L_\xi]$ where $d$ is the divisibility of the euler class of $\xi$.
 
 To do this, first we will fix a trivialization of $\TM$. Note that, Pontyragin--Thom construction works for any trivialization, but we would like to use a convenient one. Suppose $\V_1$ be the Reeb vector field of $\xi$. Now we choose a Riemannian metric such that $\V_1$ is positively orthogonal to $\xi$ with respect to this metric. Thus $\V_1$ defines $\xi$ in $\M$. To avoid ambiguity, from now on we will call the contact structure $\xi$,  $\xi_{\V_1}$ and start making alterations to $\xi_{\V_1}$ which do not affect $\xi$ or $\xi'$. Next choose $\V_2$ in the following way:
 \begin{enumerate}
 \item Choose $\V_2$ to be the tangent vector field along $\Li_i$ , if $\rot(\Li_i)$ is even.
 \item Choose $\V_2$ to be the tangent vector field along $\Li_i$ with an extra negative twist with respect to the fixed Seifert surface of the component, if $\rot(\Li_i)$ is odd.
 \end{enumerate}
 Observe that the tangent vector field $\V_2$ along $\Li=\Li'$ agrees as all the components have same $\rot$ ($\rot$ measures the winding of the tangent vector field along the component)
 Notice that as we know $\xi$ in $N$, we can extend $\V_2$ to all of $N$. Now we need to extend $\V_2$ to all of $\M$. In general, this might not be possible. The relative Euler class $e(\xi_{V_1}, \V_2)$ is the obstruction to this extension. So our goal is to make this obstruction vanish. 
 
 By using Lefchetz duality and Mayer--Vietoris sequence, we have  
 \begin{equation*}
 \label{eq:hom}
  \Ho^2(\X,\partial \X;\Z)\simeq \Ho_1(\X;\Z)\simeq \Ho_1(\M)\oplus \Z^n
  \tag{1}
  \end{equation*}
 where each of the $\Z$ factors are generated by the meridian of the link components.
 Now the relative Euler class $e(\xi_{\V_1},\V_2 )$ lives in $\Ho^2(X,\partial \X;\Z)$. By \fullref{eq:hom}, it has $n+1$ components. As the splitting suggests one can check that the relative Euler class of $\xi_{\V_1}$ relative to $\V_2$ on $\partial \X$ is computed as its evaluation on absolute chains in $\X\subset \M$  and its evaluation on the Seifert surfaces of $L_i$. For the first part, the evaluation is determined by the evaluation of $e(\xi_{V_1})$ on surfaces in $\M$. Now as $\xi_{\V_1}$ is a contact structure, it is an even class. On the other hand, by our choice of $\V_2$, 
 \[ \langle e(\xi_{\V_2},),[\Sigma_i]\rangle= \rot(\Li_i) \quad \text{or}\quad \rot(\Li_i)+1\]
 In both the cases, this is always even for each $i$. So the relative Euler class is a $n+1$ vector with every co-ordinate even. Let us rename this as $\alpha$. Next we will apply half Lutz twist to alter the relative Euler class.  Now choose a transverse knot $\K$ in $\X$ (that is  $[\K] \in \Ho_1(\X,\Z)$) such that $PD[\K]=\frac{1}{2}(\alpha)$ (We can always find such knot). If we apply half Lutz twist in $\X$ along $\K$, we get a new contact structure $\xi_{\V_2'}$ such that\[
 e(\xi'_{\V_1},\V_2)- e(\xi_{\V_1},\V_2)=-2PD[\K] \]
 By our choice of $\K$, $ e(\xi'_{\V_1'},\V_2)$ becomes zero. Thus we can extend $\V_2$ as a section of $\xi'_{\V_1}$ on all of $\X$. Now choose an almost contact structure $J$ on $\M$ and say $\V_3=J\V_2$. We use the vector fields $-\V_1, \V_2, \V_3$ to trivialize $\TM$ and $\mathrm{TX}$. Notice here $\V_1$ is mapped to the south pole $p^*$. We will call this trivialization $\tau$.

Using this trivialization, we find framed links ($\Li_\xi, f_\xi)$ and ($\Li_{\xi'}, f_{\xi'})$ associated to $\xi$ and $\xi'$ by Pontyragin--Thom construction on $\X$. As $\M$ is trivialized by $\tau$, both $\Li_\xi$ and $\Li_{\xi'}$ are oriented cycles. Next we need to show that $\Li_\xi$ and $\Li_{\xi'}$ are homologous in $\X$. As $\Ho_1(\X,\Z)$ splits in $n+1$ components,we need to check if they agree in each of them. First we will show they agree in $\Ho_1(\M,\Z)$. Now notice, $\V_1$ is the vector field that defines $\xi$ in $\N$ and also it is mapped to the south pole. So we can define a map from $\N$ to $\Sp^2$ where $\N$ is collapsed to the south pole $p^*$. Now we can extend the map $f_\xi$ in the following way:
\[F_\xi(x)=
\begin{cases}
             f_\xi(x)\quad \text{if}\quad x\in X\\
             
             p^*      \qquad \ \text{if}\quad x\in \N
 \end{cases}
 \]            
Now $F^{-1}(p)=f^{-1}(p)=\Li_\xi$. Similarly for $\Li_{\xi'}$. Thus $\Li_\xi$ and $\Li_{\xi'}$ are also associated to $\xi$ and $\xi'$ in $\M$. Now as $\xi$ and $\xi'$ are homotopic as plane fields in $\M$, the components must agree in $\Ho_1(\M,\Z)$.

Next we need to verify if $\Li_\xi\cap\Sigma_i= \Li_\xi'\cap\Sigma_i$ for each $i$. Note that here we can take the same Seifert surfaces for  each link components $\Li_i$ and $\Li_i'$ as they are related by the ambient isotopy. As the tangent vector $\V_2$ gives the framing to the link $\Li_\xi$ (as framing of $\Li_\xi$ is given by  the pull back of $T_p\Sp^2$ and this is exactly equal to $\xi$ along $\Li_\xi$), we have
\[ \langle e(\xi,\V_2), \Sigma_i\rangle=\Li_\xi\cap\Sigma_i.\] 
Same argument works for $\Li_{\xi'}$.
 Now if $\rot(\Li_i)$ is even, the definition of $\V_2$ gives us $\rot(\Li_i)=\langle e(\xi,\V_2),\Sigma_i\rangle$. 
Thus if $\rot(\Li_i)$ is even, we have, 
\[ \Li_\xi\cap\Sigma_i=\langle e(\xi,\V_2),[\Sigma]\rangle=\rot(\Li_i)=\rot(\Li_i')=\langle e(\xi',\V_2),[\Sigma]\rangle=\Li_{\xi'}\cap\Sigma_i\]
Similarly for $\rot(\Li_j)$ odd, 
\[ \Li_\xi\cap\Sigma_i=\langle e(\xi,\V_2),[\Sigma]\rangle=\rot(\Li_j)+1=\rot(\Li_j')+1=\Li_\xi'\cap\Sigma_i.\] Thus $\Li_\xi$ and $\Li_{\xi'}$ are homologous in $\Ho_1(X,\Z).$

Next we want to show that the framing differs by $2d([\Li_\xi])$. Now notice that $\xi$ and $\xi'$ are homotopic as plane fields in $\M$. Thus the framings of $\Li_\xi$ and $\Li_\xi'$ associated to $\xi$ and $\xi'$ must differ by $d(\xi)$ where $d(\xi)$ is the divisibility of $e(\xi$) \cite{gompf}. In other words, its the same as the divisibility of the Poincar\'{e} dual of $e(\xi)$.  We will show this is exactly $2d[L_\xi]$.
We know $\xi=f_\xi^*(TS^2)$.
\[e(\xi)=e(f_\xi^*(TS^2))= f_\xi^*(e(TS^2))= f_\xi^*(2[S^2]).\] Now $p=PD[S^2]$ as $p$ is a regular value. So \[f_\xi^*(2[S^2])= f_\xi^*(2PD[p])=2PD(f_\xi^{-1}(p))=2[\Li_\xi].\] For the second equality check \cite{geiges}. So the framing differs by $2d[L\xi]$.
Thus by \fullref{lemma:pont}, $\xi_{M\setminus N}$ and $\xi_{M\setminus N}'$ are homotopic rel boundary.
 \end{proof}
 \begin{proof}[Proof of \fullref{thm:main}]
 As $\Li$ and $\Li'$ are loose, they have overtwisted complements. Now by Eliashberg's classification of overtwisted contact structures we know that isotopy classes of overtwisted contact structures are in one to one correspondence with the homotopy class of plane fields \cite{eli}. Thus if each of the components of $L$ and $L'$ have same Thurston--Benniquin and rotation number, by \fullref{lemma:homotopic} they have contactomorphic complements rel boundary. As we can extend this contactomorphism  over the standard  neighborhood of $\Li$ (disjoint union of solid tori), this proves $\Li$ and $\Li'$ are coarsely equaivalent.
 \end{proof}
 \begin{corollary}
 \label{cor:transverselink}
 Suppose $\T$ and $\T'$ are two topologically isotopic loose $n$-component transverse links with each of their components being null-homologous (i.e each of the components bounds a Seifert surface). Fix these Seifert surfaces and with respect to these surfaces suppose $\slk(\T_i)=\slk(\T'_i)$,  then $\T$ and $\T'$ are coarsely equivalent.
 \end{corollary}

\begin{proof}
Suppose $\T$ and $\T'$ be two loose transverse links with each of their components being nullhomologous (i.e each component bounds a Seifert surface) and $\slk(\T_i)=\slk(\T_i')$ for each $i$. Now we Legendrian realize $\T$ and $\T'$ component by component and call them $\Li$ and $\Li'$.  We can do the Legendrian approximation in a small enough neighborhood so that the Legendrian links remain loose. After this step, 
 we can have the following two cases:

\subsubsection*{\bf Case 1}Suppose $\tb(\Li_i)=\tb(\Li_i')$ and $\rot(\Li_i)=\rot(\Li_i')$ for all $i$. Then we have two loose Legendrian links with each component null homologous and with same classical invariants. Thus by \fullref{thm:main}, they have contactomorphic complements. Now we take the transverse push-off of $\Li$ and $\Li'$. As transverse push-off is well-defined, we get back $\T$ and $\T'$. This proves that $\T$ and $\T'$ are coarsely equaivalent.
\subsubsection*{\bf Case 2}Suppose $\tb(\Li_j)\neq\tb(\Li_j')$ for some $j$. We may assume $\tb(\Li_j)>\tb(\Li_j')$. So we start by negatively stabilizing $\Li_j$. As we can do a negative stabilization in a small enough Darboux ball, this does not effect any other link component and thus without changing the transverse link type. So we can negatively stabilize each of the link components locally one by one till $\tb(\Li_i)=\tb(\Li_i')$ for each $i$. As $\slk(\T_i)=\slk(\T_i')$, we must have $\rot(\Li_i)=\rot(\Li_i')$ for each $i$ as well. So we are back in case 1.
\end{proof}
 
 \section{Links and open book decomposition}
 \label{sec:openbook}
 In this section, we extend the idea of support genus of a Legendrian knot \cite{ona} to the support genus of a link and prove that every coarse equivalence class of loose null-homologous Legendrian links have support genus zero.
 
  We can always  associate a Legendrian link in $(\M,\xi)$ with an open book supporting the underlying manifold by including the link in the $1$-skeleton of the contact cell decomposition of the contact manifold. Thus we define the support genus of a Legendrian link in $(\M,\xi)$ as follows: 
 
 \begin{definition}
 The support genus $\sg(\Li)$ of a Legendrian link $\Li$ in a contact $3$-manifold $(\M,\xi)$ is the minimal genus of a page of the open book decomposition of $\M$ supporting $\xi$ such that $\Li$ lies on the page of the open book and the framings given by $\xi$ and the page agree.
 \end{definition}

  In \cite{ona}, Onaran proved the following theorem.
 \begin{theorem}
 Any link in a 3-manifold $\M$ is planar.
 \end{theorem}
 The above theorem tells us that that any link in $\M$ can be put on a planar open book $(\B,\Sigma, \phi)$ for $\M$. For details of the proof see \cite{ona}. 
 \begin{figure}[!htbp]
 \centering
 \includegraphics[scale=0.1]{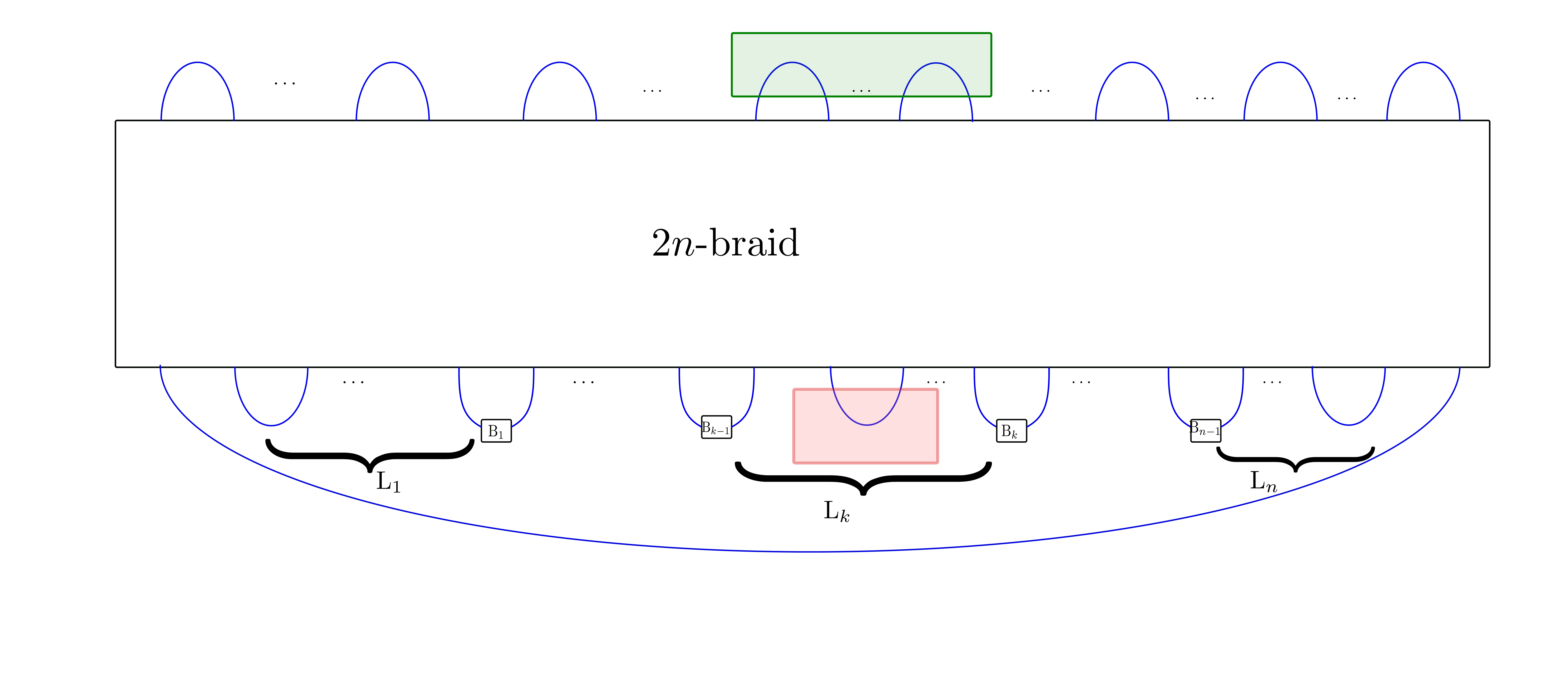}
 \caption{Page of a planar open book where the link lies. The blue outline shows the outer boundary component of the punctured disk. The box depicts the boundary area where we want to do the stabilization or detabilization of $\Li_k$.}
 \label{fig:braid}
 \end{figure}
 Now before we proceed to the main theorem of this section, we will need the following lemmas.
  \begin{lemma}
\label{lemma:Positive_stab_Leg}
Suppose $\Li$ be a Legendrian link sitting on a planar open book as shown in \fullref{fig:braid}. Then positive/negative stabilization of any of the link component $\Li_i$ can be done fixing the Legendrian isotopy type of the other link components.
\end{lemma}
\begin{figure}[!htbp]
\centering
\includegraphics[scale=0.2]{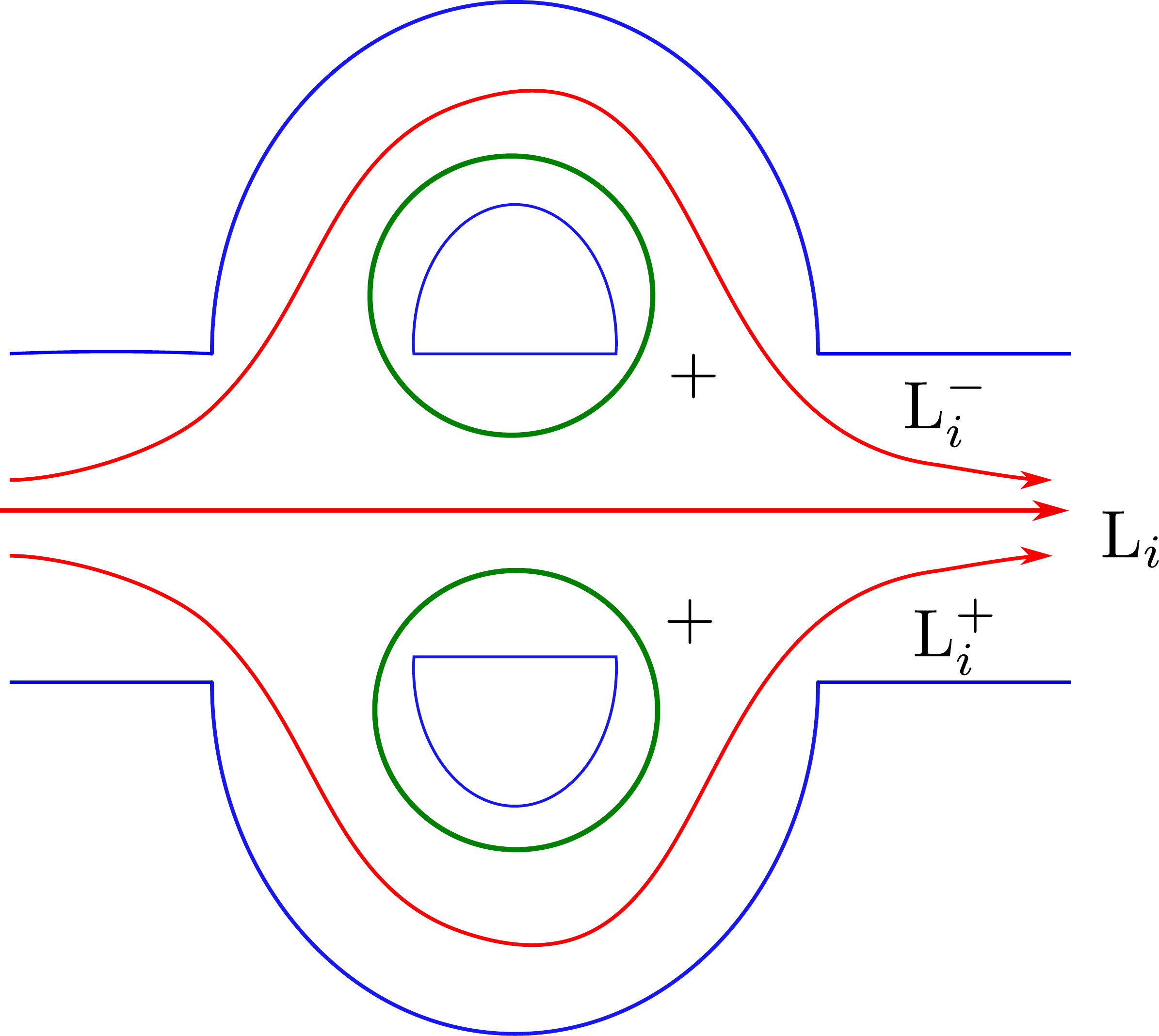}
\caption{Positive and Negative stabilization of the link sitting on the page of an open book.}
\label{fig:Positive_stab_Leg}
\end{figure}
\begin{proof}

 Suppose $\Li$ be a Legendrian link sitting on the page of a planar open book. Fix an orientation of the link. Suppose $\B_i$ is the outer most binding component. Now choose a particular region of $\B_i$ which is closest to $\Li_i$ and far from other components. The shaded region in \fullref{fig:braid} shows us where we will do the stabilizations. We do a positive stabilization along $\B_i$ and push the link component $\Li_i$ along the attaching 1-handle as shown in \fullref{fig:Positive_stab_Leg}. We call it $\Li'_i$.  by our choice of attaching region, this operation is local and thus does not affect any other link component sitting on the page of the open book. Clearly $\D$ is a disk with $\tb=-1$ and a single dividing curve. Thus we assume it to be convex. Therefore, $\Li'_i$ is the stabilization with $\D$ being the stabilizing disk. Also $\D$ can be thought as bypass disk along $\Li'$. 
 \begin{figure}[!htbp]
 \centering
 \includegraphics[scale=0.25]{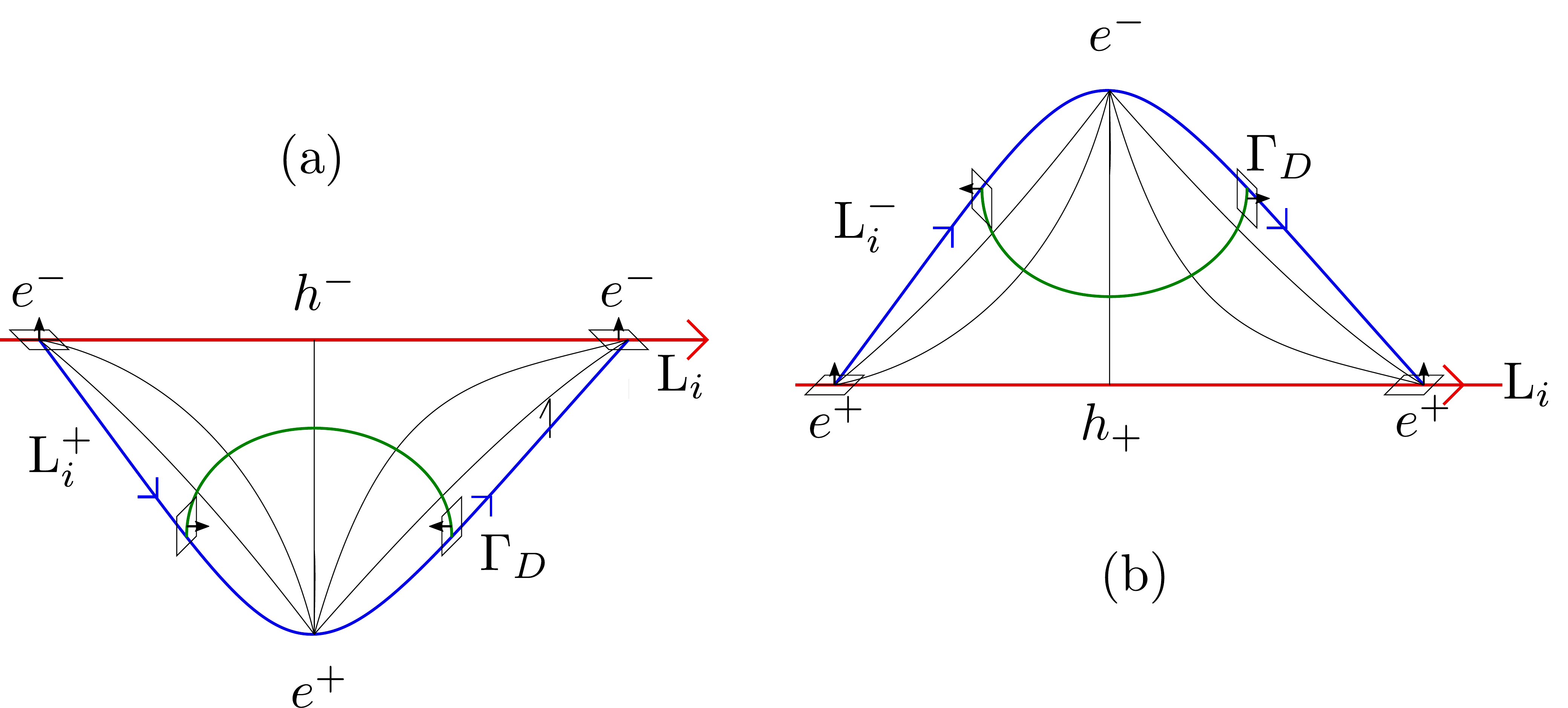}
 \caption{The positive and negative stabilization of $\Li_i$ and the signs of bypass disks.}
 \label{fig:bypass_sign}
 \end{figure}
 The sign of the stabilization will depend on the orientation of the boundary of the disk.. The orientation of the boundary of the disk is inherited by the Legendrian knot $\Li'$. The sign of the singularity of $D_\xi$ is determined by the contact planes. We will call a singularity along $\partial\D$ positive or negative according to if the contact plane takes a right handed or a left handed turn along $\partial\D$. See \fullref{fig:bypass_sign}. Now clearly we have chosen to do this operation away from the other link components. Thus all other link components remain unaltered during the operation and so are their Legendrian knot types. Observe that, $\Li_i$ has a fixed orientation. So we can perform any number of positive or negative stabilization of any link component away from the other components.
 
 \end{proof}

The next lemma tells us that de-stabilization of any component of a loose link can be done in the complement of other components.
 \begin{lemma}
 \label{lemma:neg_stab}
 Suppose $\Li$ be a link sitting on the page of a planar open book $(\B, \Sigma,\phi)$ as shown in \fullref{fig:braid}. Suppose $\B_i$ be the outer most boundary component. Now suppose we do a negative stabilization of $(\B,\Sigma,\phi)$ along $\B_i$. The new open book does not support $(\M,\xi)$ and we get a new link  $\Li_{new}$ in the new contact structure. Now if we push $\Li_{new}$ along the attaching handle, this will destabilize the link component and it can be performed in a way that it does not affect the Legendrian type of any other link components. 
\end{lemma}

 \begin{figure}[!htbp]
 \includegraphics[scale=0.2]{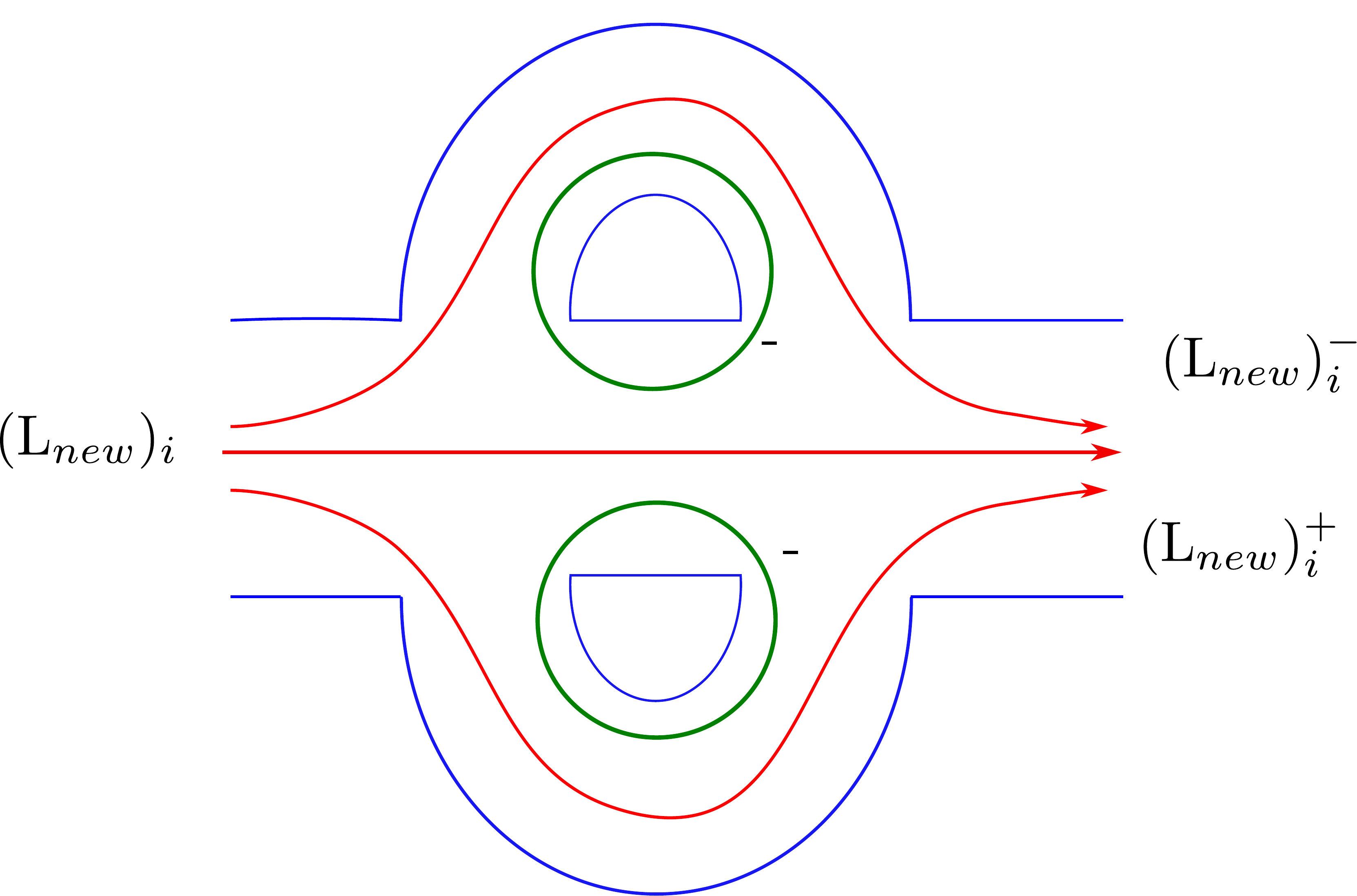}
 \caption{Negative stabilization of the open book and the de-stabilized link component sitting on the page}
 \label{fig:neg_stab_openbook}
 \end{figure}
 
 \begin{proof}
 In \cite{ona}, a similar version of this lemma has been proved for knots. We give a slightly different proof. Our proof relies on the fact that null-homologous Legendrian knots having same classical invariants are Legendrian isotopic in $\Sp^3$ if there is an overtwisted disk disjoint from them \cite{dy}.

 Suppose $\Li$ be a Legendrian link sitting on the page of a planar open book $(\B,\Sigma,\phi)$. Fix an orientation of $\Li$. Pick a link component $\Li_i$, we want to destabilize. Now we choose a particular region of the outer most boundary component near $\Li_i$ and away from all other $\Li_j$'s. This can be done as shown in \fullref{fig:braid}. 
 
  Now do a negative stabilization along that region and push the link component $\Li_i$ along the attaching 1-handle. By our choice of attaching region, this operation is away from the other link components. The new open book $(\B',\Sigma', \phi')$ doesn't support the underlying contact structure anymore. We will call the link $\Li_{new}$ in the new contact structure and show that $(\Li'_{new})_i$ is a destabilization of $(\Li_{new})_i$ as shown in \fullref{fig:neg_stab_openbook}.
 \begin{figure}[!htbp]
 \centering
 \includegraphics[scale=0.15]{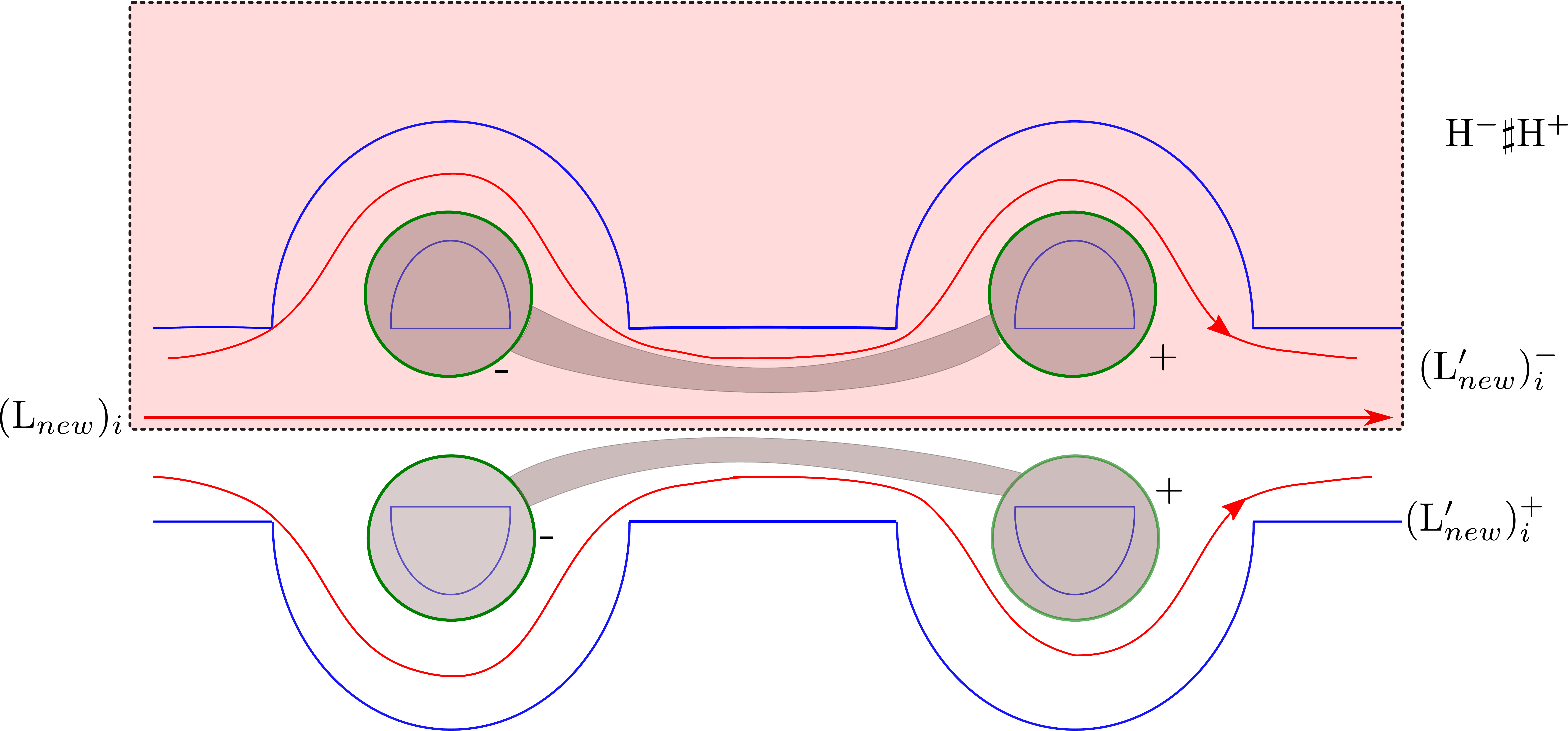}
 \caption{Negative stabilization followed by a positive stabilization of the open book near $\Li_k$ and away from other components.}
 \label{fig:destab_link}
 \end{figure}
 Here the disk $\D$ has $\tb=1$ and thus cannot be made convex. So we stabilize the open book along the same boundary component as shown in \fullref{fig:destab_link}. Now positive and negative stabilization of $(\Sigma,\phi)$ can also be thought as Murasugi summing with $(\Ho^\pm,\pi^\pm)$. Also notice $(\Ho^+,\pi^+)\connsum (\Ho^-,\pi^-)$ is an open book for $(\Sp^3,\xi_{-1})$. As the link components are identical outside the neighborhood of the boundary, we can assume the local operation to be entirely in the overtwisted $\Sp^3$. Now we push $(\Li'_{new})_i$ along the new attaching handle. And by \fullref{lemma:Positive_stab_Leg}, we get $(\Li'_{new})_i^\pm$ according to the orientation of the link component. Also  we found an overtwisted disk $\D$ in the complement of $(\Li_{new})_i$ and $(\Li'_{new})_i^\pm $. Now by \cite{dy},  $(\Li_{new})_i$ and $(\Li'_{new})_i^\pm $ must be Legendrian isotopic. As $(\Li'_{new})_i^\pm $ is a stabilization of $(\Li'_{new})_i$, clearly $(\Li'_{new})_i$ is the destabilization of $(\Li_{new})_i$. Nothing changed outside the overtwisted $\Sp^3$. Thus all other link components remain unaltered and so their Legendrian isotopy class.

 \end{proof}

 Thus \fullref{lemma:neg_stab} together with \fullref{lemma:Positive_stab_Leg} proves that if a link lies on an open book as shown in \fullref{fig:braid}, any number of positive (resp. negative) stabilization and de-stabilization of a particular link component can be done in the complement of the other link components. We will use these lemmas in the proof of our main theorem in this section.
 \begin{definition}
 Suppose $[\Li]_n$ denotes the class of all the $n$-component links with each component having fixed $\tb$ and $\rot$. For any two links in this class there exists a contactomorphism that takes one to the other. We call this \emph{the coarse equivalence class} of a link.
 \end{definition}
 \begin{theorem}
 \label{thm:sg}
 Suppose $[\Li]_n$ be the coarse equivalence class of null-homologous, loose Legendrian link in $(\M,\xi)$. Then $\sg([\Li]_n)=0$.
 \end{theorem}
\begin{proof}
 As every link is planar, we can put $\Li$ on a planar open book $(\B,\Sigma,\phi)$ for $\M$. Now $(\B,\Sigma,\phi)$ does not necessarily support the underlying contact structure. But we can always negatively stabilize the open book and assume the contact structure it supports is overtwisted and call it $\xi'$. As overtwisted contact structures can be identified using their $d_2$ and $d_3$ invariant, we start making alterations to the open book so that the invariants match with those of $\xi$. By Lutz twist and Murasugi summing in an appropriate way we can make the $d_2$ and $d_3$ invarants agree. Note that, $d_3$ invariant are additive under connected sum operation. Also none of these operations change the genus of the open book. For details of these operations check \cite{etplanar}. Now we have a planar open book which supports a contact structure whose $d_2$ and $d_3$ invariants agree with $\xi$. By Eliashberg's classification of overtwisted contact structures, these contact structures are isotopic. Next we can Legendrian realize the link on the page and call it $\Li'$. Suppose we want to realize the following classical invariants, $\tb=(t_1, t_2,\dots t_n)$ and $\rot=(r_1, r_2,\dots r_n)$.  If the classical invariants of $\Li'$ agree with that of $\Li$, we are done. Suppose not. Then we can have the following cases:
 \vspace{-3 mm}
\subsection*{Case 1}Suppose $\tb$ agrees but $\rot$ does not. Let $\Li_j$ be a link component with $\tb(\Li_j)=t_j$ and $\rot(\Li_j)=r'_j\neq r_j$. Now we will negatively or positively stabilize the link component $\Li_j$ to increase or decrease $r'_j$. We know by \fullref{lemma:Positive_stab_Leg}, this operation can be done fixing other link components. Notice, this will change $t_j$ to $t_j-1$. So we need to destabilize the link component in an appropriate way so that we do not reverse the change in $r_j$. This can be done in the following way, if we positively stabilize the link component, we will negatively destabilize it. This can be done fixing all other link components as stated in \fullref{lemma:neg_stab}. Now this will keep the $\tb$ fixed and increase $\rot$ by 2. Similarly doing a negative stabilization and a positive destabilization will keep $\tb$ fixed and decreases $\rot$ by 2. As $\tb+\rot$ is always odd for a Legendrian knot, we can achieve any possible rotation number for a link component. Now we can do this any number of time to achieve $r_j$ while fixing the Legendrian type of all other link components. Here note that, we might end up in a contact structure different from the one we started as negative stabilization alters a contact structure. But then we can always alter it by Murasugi summing with appropriate open books of $\Sp^3$. In this way, we will find a link sitting on the page of an open book supporting the contact structure $\xi$ with $\tb=(t_1, t_2,\dots t_n)$ and $\rot=(r_1, r_2,\dots r_n)$. By \fullref{thm:main}, $\Li$ must be in the same coarse equivalence class . This proves the theorem.

\subsection*{Case 2}Suppose $\tb(\Li_j)=t'_j\neq t_j$. In this case we need to stabilize or destabilize the link component $\Li_j$ to decrease and increase the $\tb$ till it agrees with $t_j$ and this can be done keeping the other components fixed by \fullref{lemma:Positive_stab_Leg} and \fullref{lemma:neg_stab}. Now we can do this local operation for all the link components one by one till we get the $\tb$ we desire. So we are in Case 1. 
 \end{proof}
 The next theorem tells us that the converse of the above theorem is not true.
 \begin{theorem}
 There are examples of non-loose links with support genus zero.
 \end{theorem}
 \begin{figure}[!htbp]
 \includegraphics[scale=0.15]{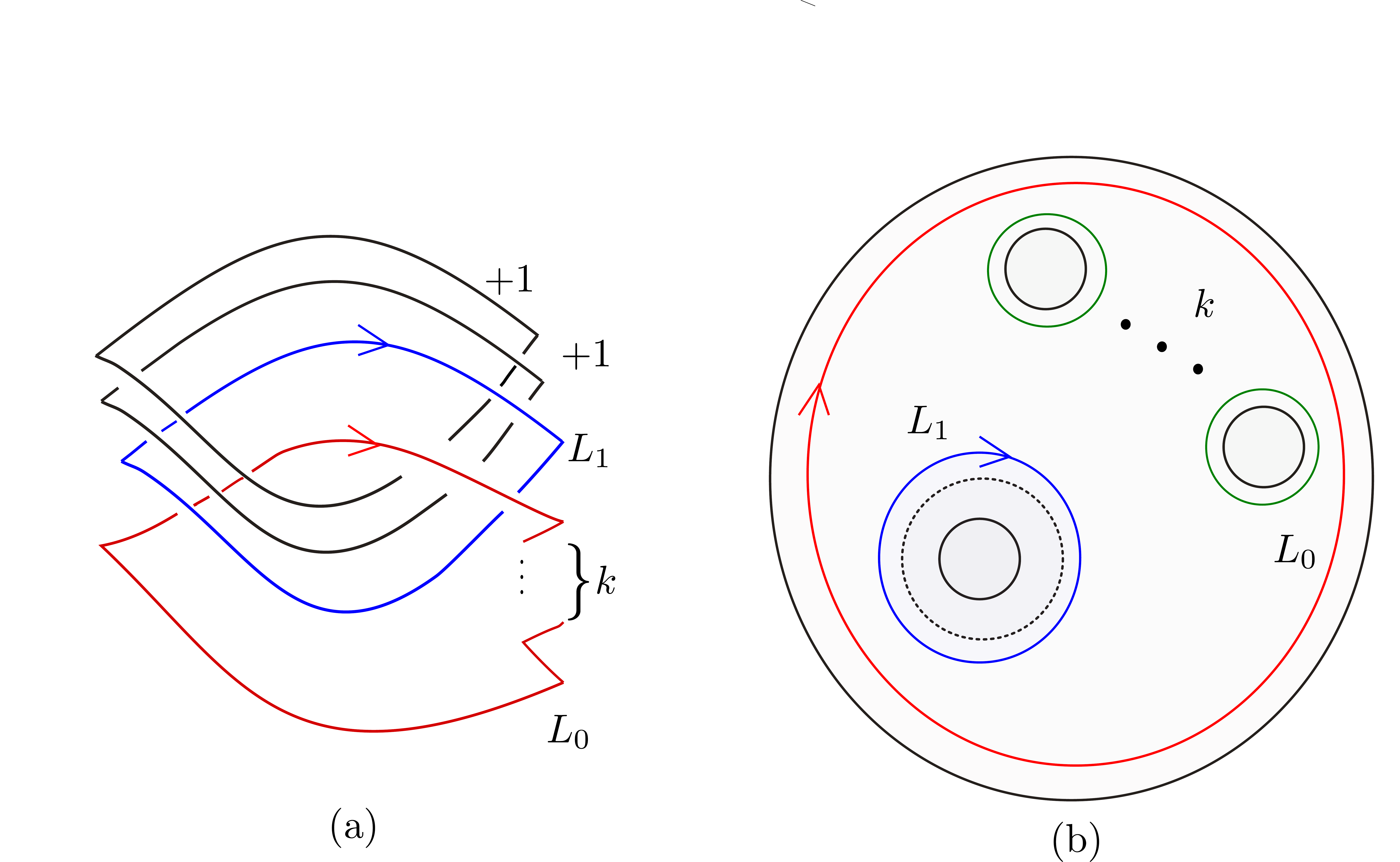}
 \caption{Example 1:(a) Non-loose Hopf link in $(\Sp^3,\xi_{-1})$. (b) Planar open book supporting the contact structure where the Hopf link sits. We do a right handed Dehn twist along the green curves and left-handed Dehn twist along the dashed one.}
  \label{fig:Hopf}
 \end{figure}
 \begin{proof}
  \fullref{fig:Hopf}(a) shows a non-loose positive Hopf link in $(\Sp^3,\xi_{-1})$. To see this, we do a -1 surgery along $\Li_1$ which will cancel one of the +1- surgeries and we will be left with one +1- surgery on $\tb=-1$ unknot in $(\Sp^3,\xistd)$ which produces the unique tight $\Sp^1\times\Sp^2$. For details, check \cite{geiona}. Next, we constructed a planar open book compatible with $(\Sp^3,\xi_{-1})$ where the non-loose Hopf link sits. We start with the annular open book that supports $(\Sp^3,\xistd)$ and used the well known stabilization method we used previously in \fullref{lemma:Positive_stab_Leg}. The monodromy of this open book can be computed from the Dehn twists coming from the stabilizations and the Dehn twists defined by the surgery curves. One of the left-handed Dehn twist coming from the +1 surgery will cancel the right handed Dehn twist of the annular open book we started with. We perform right handed Dehn twist along the solid green curves and the left handed Dehn twist along the dashed curve.  This clearly shows $\sg(\Li_0\sqcup\Li_1)=0$.
  \begin{figure}[!htbp]
  \centering
  \includegraphics[scale=0.15]{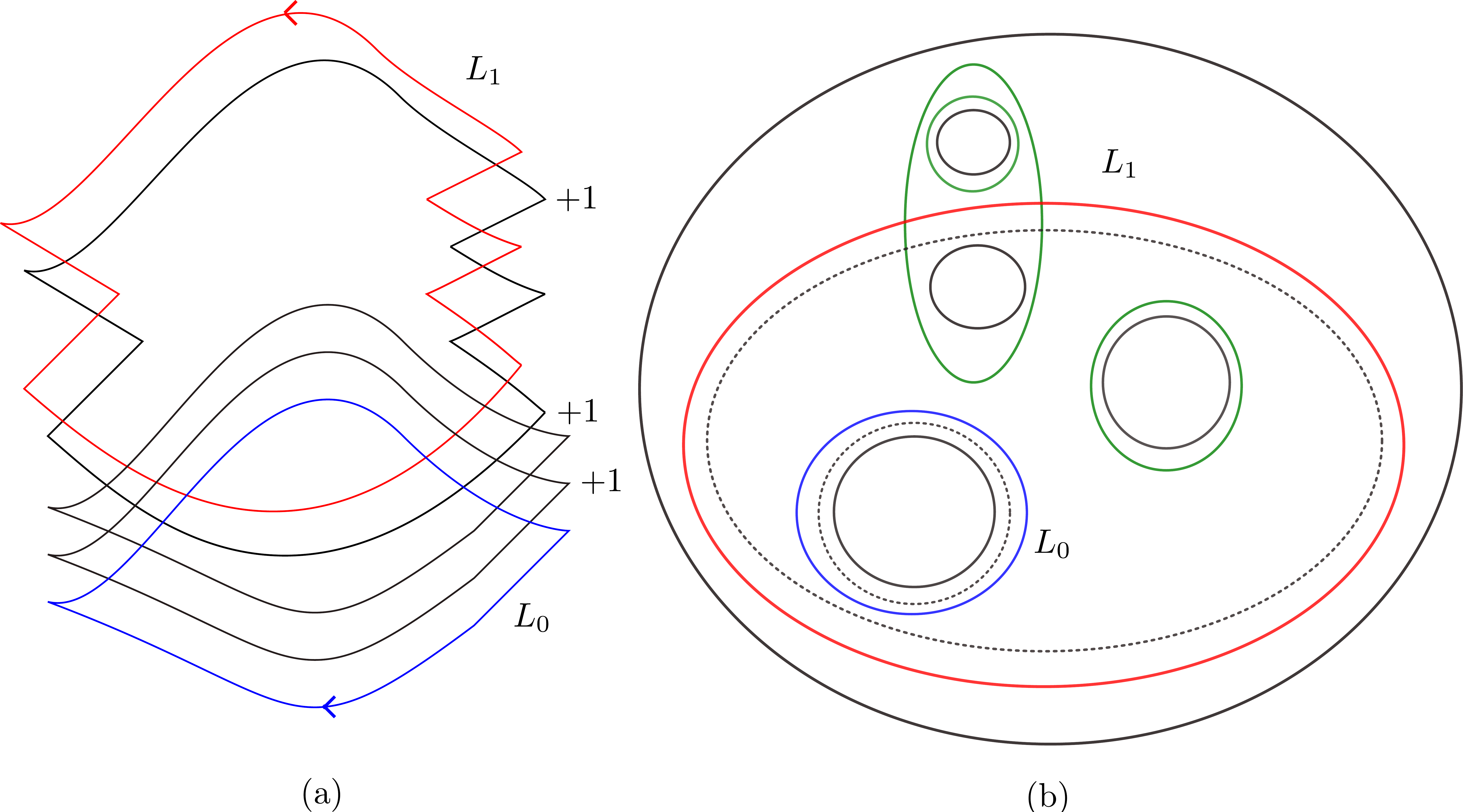}
  \caption{Example 2:(a) Non-loose Hopf link in $(\Sp^3,\xi_{-2})$. (b) A compatible planar open book with the link lying on the page.    }
  \label{fig:Hopf2}
  \end{figure}
  
   Example 2 shows a non-loose Hopf link in $(\Sp^3,\xi_{-2})$. Here a $-1$ surgery on $L_2$ and $-2$ surgery on $L_2$ gives us $(\Sp^3,\xistd$). Check \cite{geiona} for details. We produce a compatible open book for this contact structure as follows: like before we started with an annular open book supporting $(\Sp^3,\xistd)$ and used the well known stabilization method. Notice that the right handed Dehn twist coming from the annular open book again gets cancelled with the negative Dehn twist coming from one of the surgery curve. We perform right handed Dehn twist along the green curves and one left handed Dehn twist along the black dashed curves.
 \end{proof}
  \bibliographystyle{mwamsalphack}
\bibliography{references1}
\end{document}